\documentclass[12pt]{amsart}
\usepackage{amsmath,amscd,amsthm,amssymb,enumerate, url, eucal, fullpage}

\swapnumbers

\theoremstyle{plain}
\newtheorem{theorem}{Theorem}

\newtheorem{qs}[theorem]{The Quillen-Suslin Theorem}

\theoremstyle{definition}

\newtheorem{review}[theorem]{Review}

\newtheorem{dandn}[theorem]{Definitions and Notation}
\newtheorem{basics}[theorem]{Lemma}

\newcommand{\ZZ}{{\mathbb Z}}
\newcommand{\NN}{{\mathbb N}}
\renewcommand{\phi}{\varphi}

\renewcommand{\ge}{\geqslant}
\renewcommand{\epsilon}{\varepsilon}

\numberwithin{equation}{theorem}

\def\d1{\discretionary{-}{}{-}}

\begin{document}

\title{Seminar\hspace{4pt}on\hspace{4pt}Quillen's\hspace{4pt}proof\hspace{4pt}of\hspace{4pt}the\hspace{4pt}Quillen-Suslin\hspace{4pt}Theorem}

\author{Warren Dicks}
\address{Departament de Matem{\`a}tiques, \newline \indent Universitat Aut{\`o}noma de Barcelona, \newline \indent
 E-08193 Bellaterra (Barcelona),  SPAIN \newline \indent \vspace{-5pt}\null}
\email{dicks@mat.uab.cat \newline \indent  \normalfont {\emph {Home page}}:\,\,\,\,\url{http://mat.uab.cat/~dicks/}}

\null 

\maketitle
 
\vspace{-.2cm}

Let $K$ be a field, and $n$ be a positive integer.     \cite[p.\,243]{Serre1955}
remarked ``On ignore s'il existe des \mbox{$K[x_1,  \ldots, x_n]$}-modules projectifs de type fini qui ne soient
pas libres.".
Motivated by this remark, \cite{Quillen1976} and  \cite{Suslin1976} independently 
 showed that  every finitely generated, 
projective  \mbox{$K[x_1,  \ldots, x_n]$}-mod\-ule is free, a 
result that is now known as  \textit{the Quillen\d1Suslin theorem}.
 \cite{Lam2003} gave  an extensive  collection of  proofs and related results that  
were available 30 years later.

We construct an idem\-potent\d1matrix\d1based version of Quillen's proof 
by translating most of Paul Roberts'   
geometry-free proof of   Horrocks' theorem  and  translating part of
Moshe Roitman's  proof of  Quillen's extension theorem.  
 We described all of this  in the Bedford College Ring Theory Study Group in the autumn term of 1977,
 and gave the  
participants copies of the three-page handout \cite{Dicks1977}  later
cited  by      \mbox{\cite[pp.$\mkern1mu160,\mkern0mu201$]{Lam2003}}.   
The present write-up  includes minor changes and
three pages of introductory material.

 \begin{dandn} \noindent $(i).$  As Bourbaki intended, we let 
$\mathbb{N}$ denote the set of finite cardinals; that is, $\mathbb{N}:= \{0,1,2,,\ldots\}$. 

 \smallskip

\noindent $(ii).$  By a \textit{ring} $R$, we  mean an associative ring with $1$.
A left $R$-module $M$ is said to be \textit{$R$-free} (resp.\,\textit{$R$-projective},
  resp.\,\textit{finitely $R$-generated}) if  $M$ is a free (resp.\,projective,  resp.\,finitely generated)
  left $R$-module.  For any subring $R_0$ of~$R$,  if 
\mbox{$M \underset{R}{\simeq}  R \otimes_{R_0} M_0$} 
for some left $R_0$-mod\-ule~$M_0$,  
then, for any ideal $J$~of $R$, we  have 
\mbox{$(R/J) \otimes _R M  
\underset{R/J}{\simeq}   (R/J)  \otimes_{R_0} M_0.$} 
If, moreover, the natural map \mbox{$R_0 \to R/J$} is bijective, then  
\mbox{$R_0 \otimes _R M \underset{R_0}{\simeq} M_0$}; here,
 $M_0$ is  uniquely determined by~$M$  up to $R_0$-isomorphism, and,
 if $M$ is a finitely generated, projective left $R$-module,
then $M_0$ is a finitely generated, projective left $R_0$-module.

\smallskip

\noindent $(iii).$  Whenever $R$ is a commutative ring, the following will apply: \mbox{$R[x]$}  denotes 
the polynomial ring over~$R$ in one central variable denoted $x$;    \mbox{$R(x)$}  denotes the
localization of \mbox{$R[x]$} at the set of all monic polynomials in \mbox{$R[x]$};  
 \mbox{$\operatorname{Max}(R)$}  denotes the set of all
maximal proper ideals of $R$; and, for each \mbox{$\mathfrak{m} \in \operatorname{Max}(R)$}, 
  \mbox{$R_{\mathfrak{m}}$}~denotes 
the localization of $R$ at \mbox{$R{-}\mathfrak{m}$}, the complement  of $\mathfrak{m}$ in $R$. 
 Also, we shall say the following: $R$ is a \textit{domain} if $R$~is nonzero and $R$ 
has no nonzero zerodivisors;  $R$ is a \textit{Dedekind domain} if $R$ is a domain and,  
for each $\mathfrak{m} \in \operatorname{Max}(R)$, the localization $R_{\mathfrak{m}}$ is a PID;
and  $R$ is \textit{local} if  $R$ has a unique maximal proper ideal.
\qed
\end{dandn}

\vspace{-2mm}

\begin{review}\label{rev}
Let $R$ be a ring.

\smallskip

\noindent $(i).$ Let \mbox{$m,\,n \in \NN$}.  We write \mbox{$\null^m\mkern-3muR\mkern2mu^n$} to 
denote the set of all $m \times n$ matrices over $R$. 
We give \mbox{$\null^n\mkern-3muR\mkern2mu^n$} the  standard ring structure.
We set  \mbox{$R\mkern2mu^n := \null^1\mkern-3muR\mkern2mu^n$}, a free left $R$-module.
For each \mbox{$A \in \null^m\mkern-3muR\mkern2mu^n$}, right multiplication by $A$ gives an $R$-linear map
 \mbox{$R\mkern2mu^m \overset{A\,\,}{\to} R\mkern2mu^n$}, whose image will be denoted \mbox{$R\mkern2mu^m{\cdot}A$}. 
Conversely, every $R$-linear map \mbox{$ R\mkern2mu^m  \overset{\alpha}{\to}  R\mkern2mu^n$} is given by 
right multiplication by 
a unique \mbox{$A \in \null^m\mkern-3muR\mkern2mu^n$}; we shall view $A$ and $\alpha$ as 
  interchangeable. \qed

\smallskip

\noindent $(ii).$  Let  \mbox{$n \in \NN$}  and  \mbox{$E \in \null^n\mkern-3muR\mkern2mu^n$} 
such that \mbox{$E\mkern1mu^2 =E$}; we say that $E$ is  \textit{idempotent}.  We  now show that
\mbox{$R\mkern2mu^n{\cdot}E$} is a finitely generated, projective left $R$-module.

It suffices to show that    \mbox{$R\mkern2mu^n = 
R\mkern2mu^n{\cdot}E \oplus  R\mkern2mu^n{\cdot}(I_n{-}E)$} as left \mbox{$R$}-mod\-ules.
It is clear that  \linebreak\mbox{$R\mkern2mu^n = 
R\mkern2mu^n{\cdot}E +  R\mkern2mu^n{\cdot}(I_n{-}E)$}, since each \mbox{$A \in R\mkern2mu^n$} equals
  \mbox{$A{\cdot}E + A {\cdot}(I_n{-}E)$}.
To see that  \linebreak\mbox{$R\mkern2mu^n{\cdot}E \cap  R\mkern2mu^n{\cdot}(I_n{-}E) = \{0\}$},
notice that this intersection is annihilated by right multiplication by \mbox{$I_n{-}E$} and by 
right multiplication by \mbox{$E$}; hence it is annihilated by right multiplication by \mbox{$E + (I_n{-}E) = I_n$}.
The assertion is now proved.  \qed

\smallskip

\noindent $(iii).$  Let $P$ be a finitely generated, projective left $R$-module; thus,  there 
exist  an \mbox{$n \in \NN$} and a left $R$-module $Q$ such that \vspace{-1.2mm}  
 \mbox{$R^n \underset{R}{\simeq} P{\oplus}Q$}.
We   now show that there exists  an  idempotent \mbox{$E \in \null^n\mkern-3muR\mkern2mu^n$}
such that \mbox{$P \underset{R}{\simeq} R\mkern2mu^n{\cdot}E$}.  

We have a surjective $R$-linear map
\mbox{$\alpha{:\,}R\mkern2mu^n  \simeq  P{\oplus}Q \hskip-2pt\to\mkern-13mu\to P$} and
an injective $R$-linear map \mbox{$\beta{:\,}P \hookrightarrow P{\oplus}Q \simeq R\mkern2mu^n$} such that
 \mbox{$\beta {\cdot} \alpha = 1_P$}, where composition is to be read from left to right, and maps
are to be applied on the right of their arguments.   We view the $R$-linear map 
\mbox{$\alpha{\cdot}\beta{:\,} R\mkern2mu^n \to R\mkern2mu^n$}
as an  element $E$~of~\mbox{$\null^n\mkern-3muR\mkern2mu^n$}.
Then  \mbox{$E^2 = \alpha{\cdot}\beta{\cdot} \alpha{\cdot}\beta = 
\alpha{\cdot}1_P{\cdot}\beta = \alpha{\cdot}\beta = E$}; thus,
$E$~is idempotent.  Also, \mbox{$P = (R\mkern2mu^n)\alpha
 \underset{R}{\simeq}  ((R\mkern2mu^n)\alpha)\beta =  R\mkern2mu^n{\cdot}E$}.\qed

\smallskip

\noindent $(iv).$  For idempotent matrices $E$,  $F$ over $R$, \vspace{-1.4mm}
we write \mbox{$E \underset{R}{\sim} F$} to mean that
there exist matrices \mbox{$A$, $B$} over $R$  such that \mbox{$A{\cdot}B = E$} and \mbox{$B{\cdot}A = F$}. 

Let $m$ and $n$ be elements of $\NN$,   $E$  an idempotent element of 
 \mbox{$\null^m\mkern-3muR\mkern2mu^m$},
and   $F$   an idempotent element  of  \mbox{$\null^n\mkern-3muR\mkern2mu^n$}.
 We  now show   that \mbox{$E \underset{R}{\sim} F$} if and only if 
\mbox{$R\mkern2mu^m{\cdot}E \underset{R}{\simeq} R\mkern2mu^n{\cdot}F$}.

Consider first the situation where \mbox{$E \underset{R}{\sim} F$}.  Here, there exist  
   \mbox{$A \in \null^m\mkern-3muR\mkern2mu^n$}
and \mbox{$B \in \null^n\mkern-3muR\mkern2mu^m$} such that \mbox{$A{\cdot}B = E$} and \mbox{$B{\cdot}A = F$}. 
One fact that we shall use repeatedly throughout these notes is that  
 \mbox{$E{\cdot}A = A{\cdot}F$}  here,  since both sides equal \mbox{$A{\cdot}B{\cdot}A$}.
Now \mbox{$A$} may be viewed as the $R$-linear map \mbox{$R\mkern2mu^m \to R\mkern2mu^n$}  given by right 
multiplication by~$A$.  The \mbox{$R$}-submodule \mbox{$R\mkern2mu^m{\cdot}E$} of \mbox{$R\mkern2mu^m$}
is carried by~$A$ to \mbox{$(R\mkern2mu^m{\cdot}E){\cdot}A =   (R\mkern2mu^m{\cdot}A){\cdot}F \subseteq  
R\mkern2mu^n{\cdot}F$}, giving an $R$-linear map \mbox{$\alpha{:\,}R\mkern2mu^m{\cdot}E \to R\mkern2mu^n{\cdot}F$}.  
Similarly, $B$~gives an $R$-linear map  \mbox{$\beta{:\,}R\mkern2mu^n{\cdot}F \to R\mkern2mu^m{\cdot}E$}.    The total 
effect of applying $\alpha$ and then $\beta$ to \mbox{$R\mkern2mu^m{\cdot}E$} is to right multiply by $A{\cdot}B = E$, 
which fixes all elements of \mbox{$R\mkern2mu^m{\cdot}E$}.  Similarly $\beta{\cdot}\alpha$ acts as the identity 
endomorphism of \mbox{$R\mkern2mu^n{\cdot}F$}.  Hence, 
\mbox{$R\mkern2mu^m{\cdot}E \underset{R}{\simeq} R\mkern2mu^n{\cdot}F$}.

Finally, suppose that we have an $R$-linear isomorphism  
 \mbox{$R\mkern2mu^m{\cdot}E \overset{\alpha\,}{\to} R\mkern2mu^n{\cdot}F$}. 
The composite \mbox{$R\mkern2mu^m \xrightarrow{E} R\mkern2mu^m{\cdot}E \xrightarrow{\alpha} 
 R\mkern2mu^n{\cdot}F \subseteq R\mkern2mu^n$} acts as right multiplication by a unique
\mbox{$A \in \null^m\mkern-3muR\mkern2mu^n$}.  The composite 
\mbox{$R\mkern2mu^n \xrightarrow{F} R\mkern2mu^n{\cdot}F \xrightarrow{\alpha^{-1}} 
 R\mkern2mu^m{\cdot}E \subseteq R\mkern2mu^m$} acts as right multiplication by a unique
\mbox{$B \in \null^n\mkern-3muR\mkern2mu^m$}.  The  composite \mbox{$A{\cdot}B$}
equals the composite \mbox{$R\mkern2mu^m \xrightarrow{E} R\mkern2mu^m{\cdot}E  \subseteq R\mkern2mu^m$}, which  
is right multiplication by~$E$;  hence,  \mbox{$A{\cdot}B = E$}.  Similarly \mbox{$B{\cdot}A = F$}.
Thus,  \mbox{$E \underset{R}{\sim} F$}. \qed

\smallskip

\noindent $(v).$ It is now clear that \mbox{$\underset{R}{\sim}$} is an equivalence relation on the set of all 
idempotent (square) matrices over $R$.  Also, the  set of equivalence classes is in bijective correspondence 
with the collection of all isomorphism classes of all finitely generated, projective left $R$-modules. 
 \qed\,\,$\Box$ 
\end{review}

\begin{basics}\label{basics} Let $R$ be a  local  commutative ring.  Let 
  $\mathfrak{m}$  denote  the maximal proper ideal of~$R$.  Set \mbox{$\overline R := R/\mathfrak{m}$}, the
residue field, and  let \mbox{$R \to \overline R$}, \mbox{$r \mapsto \overline r$},  denote the quotient map.
We extend the overline notation  to polynomials and matrices. 
We identify \mbox{$R(x)/(\mathfrak{m}{\cdot} R(x))$}  with~\mbox{$\overline R(x)$}, 
which is the field of fractions of \mbox{$\overline R[x]$}. 

\smallskip

\noindent $(i)$.  Each element of $R(x)$ equals \mbox{$\frac{f}{g}$}   
for some \mbox{$f,\, g \in R[x]$} with $g$~monic.  The  \vspace{.5mm}  polynomial  division algorithm says that,
for any \mbox{$f,\, g \in R[x]$} with $g$~monic, there exist 
 unique  \mbox{$q,\,r \in R[x]$} such that 
\mbox{$f = g{\cdot}\,q +r$} and  \mbox{$\operatorname{deg}r < \operatorname{deg}g$}, 
where \mbox{$\operatorname{deg} 0 := -\infty$};
here, $q$ \vspace{.3mm} is uniquely determined by~\mbox{$\frac{f}{g}$}, and we call
$q$ the \textit{polynomial part of} $\frac{f}{g}$, denoted  \mbox{$[\frac{f}{g}]$}.  The map
 \mbox{$R(x) \to R[x]$}, 
\mbox{$\frac{f}{g} \mapsto [\frac{f}{g}]$}, is an $R$-linear retraction.  We denote the kernel by 
$R(x)_\circ:=\{\frac{f}{g}\in R(x):\operatorname{deg}f < \operatorname{deg}g\}$, 
and  obtain an $R$-module  decomposition  \mbox{$R(x)= R[x]\oplus R(x)_\circ$}.  \qed

\smallskip

\noindent $(ii)$.  For  $r \in R$, if $\overline r \ne \overline 0$, then $r \notin \mathfrak{m}$, hence 
$R\mkern2mur$ is not a proper ideal of $R$,  and  hence  $r$ is invertible in~$R$. \qed

\smallskip

\noindent $(iii)$.  Let \mbox{$n \in \NN$}.  We now show that, for each invertible 
\mbox{$A \in \null^n\mkern-3mu \overline R(x)\mkern0mu^n$}, there exists an invertible 
\mbox{$B \in  \null^n\mkern-3mu R(x)\mkern0mu^n$} such that \mbox{$\overline B = A$}; we say that 
the invertible matrix over \mbox{$\overline R(x)$}
\textit{lifts back} to some invertible matrix over $R(x)$.

By  $(ii)$ above, each invertible element of the field $\overline R$ lifts back to some invertible 
element of~$R$.  It then follows that each invertible element of the field \mbox{$\overline{R}(x)$}
 lifts back to some invertible element of~\mbox{$R(x)$}.  
It then follows that each invertible elementary matrix over \mbox{$\overline{R}(x)$}
 lifts back to some invertible elementary matrix over~\mbox{$R(x)$}. 
Over the field \mbox{$\overline{R}(x)$}, each invertible matrix is a product of 
invertible elementary matrices, and, hence, lifts back to some invertible  matrix over~\mbox{$R(x)$}. \qed 

\smallskip

\noindent $(iv)$. Let $W$ be an $R$-module, and  $W'$ be  an $R$-submodule of $W$ such that, firstly,
\mbox{$W/W'$}~is finitely $R$-generated, and, second,
\mbox{$\mathfrak{m}{\cdot}W  {+} W'  = W$}. 
We shall now show that \mbox{$W' = W$}. 
(This is  the local commutative case of one form of Nakayama's Lemma.)

Since \mbox{$W/W'$}~is finitely $R$-generated,  
there exists a smallest \mbox{$n \in \NN$}  such that there exist   \mbox{$w_1,  \ldots, w_n \in W$} 
such that  \mbox{$ R{\cdot}w_1{+}   \cdots {+} R{\cdot}w_n    +W'   = W$}. We want to show that \mbox{$n = 0$}.
\linebreak Suppose that \mbox{$n \ge 1$}; it suffices to obtain a contradiction.
Set \mbox{$W'':= R{\cdot}w_2{+} \cdots {+} R{\cdot}w_n  {+} W'$.}  Then  
\mbox{$W = R{\cdot}w_1 {+} W''$}. By the 
minimality  of~$n$, \mbox{$W'' \ne W$}; hence, \mbox{$w_1 \not\in W''$}.
Now \linebreak  \mbox{$w_1 \in W = \mathfrak{m}{\cdot}W {+} W'   =  
  \mathfrak{m}{\cdot}(R{\cdot}w_1 {+} W'') {+} W'  \subseteq  \mathfrak{m}{\cdot}w_1  {+} W''$}.
Hence,  there exists some \mbox{$r  \in \mathfrak{m}$} such that \mbox{$w_1   \in  r {\cdot} w_1 {+} W''$}, and than
  \mbox{$(1{-}r){\cdot}w_1 \in W''$}.  
 By~$(ii)$ above, \mbox{$1{-} r$} is invertible in~$R$, since 
\mbox{$\overline {1{-}r} = \overline 1 {-} \overline 0 \ne \overline 0$}.
  Hence,   \mbox{$w_1  \in (1{-}r)^{-1}{\cdot} W'' \subseteq W''$},  
which contradicts \mbox{$w_1 \not\in W''$}.  \qed

\smallskip

\noindent $(v)$.    We now give an elementary proof that if the local commutative ring $R$ is a  PID,  
then   $R(x)$   is a PID.  (This is the local case of \cite[Lemma 2]{Quillen1976}.)

If $R$ is a field, then $R(x)$ is the field of fractions of $R[x]$, which is clearly a PID;
thus, we may  assume that the local PID $R$ is not a field.  Here \mbox{$\mathfrak{m} = R{\cdot}p$} for some nonzero~\mbox{$p \in R$}. 
Then $p$ is irreducible in $R$, and each nonzero element of $R$ has a unique expression 
as~\mbox{$u{\cdot}p^n$} where \mbox{$n \in \NN$} and \mbox{$u \in R{-}R{\cdot}p$}.
  Each nonzero element of  \mbox{$R(x)$} has a unique  expression as~\mbox{$f{\cdot}p^n$}
where \mbox{$n \in \NN$} and \mbox{$f \in R(x){-}R(x){\cdot}p$}.  
We identified \mbox{$R(x)/(R(x){\cdot}p) = \overline R(x)$}, the field of fractions 
of~\mbox{$\overline R[x]$}; thus,  \mbox{$R(x){\cdot}p \in \operatorname{Max}(R(x))$}.

Let $J$ be an arbitrary ideal of $R(x)$.  We wish to show that $J$ is a principal ideal of $R(x)$.  
We may assume that \mbox{$J \ne \{0\}$}.    
Let $K$ denote the field of fractions of $R$. Then \mbox{$R(x)$} lies in the field of fractions \mbox{$K(x)$} of  
\mbox{$K[x]$}.  The subring \mbox{$R(x){\cdot}K$} of \mbox{$K(x)$} is a PID, since it
is the  localization of the PID \mbox{$K[x]$} at the set of monic polynomials in \mbox{$R[x]$}.  
As \mbox{$J{\cdot}K$} is a nonzero ideal of the PID \mbox{$R(x){\cdot}K$}, there exists  some nonzero 
\mbox{$f \in J{\cdot}K$} such that \mbox{$J{\cdot}K =  R(x){\cdot}K{\cdot}f$}.  We may  arrange that 
\mbox{$f \in J {-} J{\cdot}p$}.
It suffices to show that \mbox{$J = R(x){\cdot}f$}.  It is clear that \mbox{$R(x){\cdot}K{\cdot}f \supseteq J   \supseteq R(x){\cdot}f$}.
Consider any nonzero element of \mbox{$J$} and express it as \mbox{$g{\cdot}f$} for some \mbox{$g \in R(x){\cdot}K$};
it suffices to show that \mbox{$g \in R(x)$}.  We may write \mbox{$g = p^{-n}{\cdot}g'$}, where \mbox{$n \in \ZZ$} 
  and \mbox{$g' \in R(x){-}R(x){\cdot}p$}; it suffices to show that \mbox{$-n \ge 0$}.  Suppose that
\mbox{$n \ge 1$}; it suffices to obtain a contradiction.  
Since \mbox{$R(x){\cdot}p \in \operatorname{Max}(R(x))$},
we see that \mbox{$R(x){\cdot}p + R(x){\cdot}g' = R(x)$}.  Thus,  $p$~and (hence) $p^n$ become invertible in 
\mbox{$R(x)/(R(x){\cdot}g')$}.
As \mbox{$g'{\cdot}f = p^{n}{\cdot}g{\cdot}f \in p^{n}{\cdot}J$}, we see that \mbox{$g'$, $p^n$,} 
and (hence) $1$ all lie in the ideal
  \mbox{$\{h \in R(x) : h{\cdot}f \in p^{n}{\cdot}J  \}$} of $R(x)$. Now
\mbox{$1{\cdot}f \in p^{n}{\cdot}J$}, and this  contradicts 
\mbox{$f \not\in J{\cdot}p$}.   \qed\,\,$\Box$ 
\end{basics}

\null \vspace{-9mm}  

 By \cite[Theorem\,3]{Quillen1976}, if $R$ is a   commutative ring     and     $P$ is a finitely
generated, projective $R[x]$-module such that \mbox{$R(x){\otimes}_{R[x]}P$} is $R(x)$-free, 
then $P$ is $R[x]$-free.    
Using \cite[Theorem\,1]{Horrocks1964}, \mbox{M.\hspace{3pt}Pavaman\hspace{3pt}Murthy}
 \cite[\mbox{Lemma on p.\hspace{1pt}24}]{Bass1973} had proved the case where $R$ is  noetherian and local,
which   is already sufficient for our main purpose.  Here,  we shall prove  the case where $R$ is local.

\begin{theorem}\label{th:Horr} If $R$ is a  local commutative ring     and     $P$ is a finitely
generated, projective $R[x]$-module such that \mbox{$R(x){\otimes}_{R[x]}P$} is $R(x)$-free, then $P$ is $R[x]$-free.
\end{theorem}

\begin{proof}[Proof $($Paul Roberts$)$] We use the notation of 
Lemma~\ref{basics} above.

By Review~\ref{rev}$(iii)$ above, there exist an \mbox{$n \in \NN$} and an idempotent  
\mbox{$E \in \null^n\mkern-3muR[x]\mkern0mu^n$} such 
that \mbox{$P \underset{R[x]}{\simeq} R[x]\mkern2mu^n{\cdot}E$}; hence,
\mbox{$R(x) \otimes_{R[x]}\mkern-3muP  \underset{R(x)}{\simeq} R(x)\mkern2mu^n{\cdot}E$}. 
By Review~\ref{rev}$(iv)$ above, since \mbox{$R(x) \otimes_{R[x]} P$} is $R(x)$-free, 
  there exist an \mbox{$m  \in \NN$}, an \mbox{$A \in \null^n\mkern-3muR(x)\mkern0mu^{m}$}, 
and a \mbox{$B \in \null^{m}\mkern-3muR(x)\mkern0mu^n$}
such that \mbox{$A{\cdot}B = E$} and \mbox{$B{\cdot}A = I_{m}$}. 
Also  by Review~\ref{rev}$(iv)$ above, to show that $P$~is $R[x]$-free it suffices to show that
 there exist   an
\mbox{$A'' \in \null^n\mkern-3muR[x]\mkern0mu^m$} and a \mbox{$B'' \in \null^m\mkern-3muR[x]\mkern2mu^n$}   such that
\mbox{$A''{\cdot}B'' = E$} and \mbox{$B''{\cdot}A'' = I_m$}.

Here,  \mbox{$\overline E \in \null^n\mkern-2mu\overline R[x]\mkern0mu^n$}.
Now $\overline R[x]$ is a PID; hence,  finitely generated, projective 
\mbox{$\overline R[x]$}-modules are \mbox{$\overline R[x]$}-free. 
By Review~\ref{rev}$(iv)$ above, there exist an \mbox{$m' \in \NN$}, a 
 \mbox{$\hat C \in \null^n\mkern-3mu\overline R[x]\mkern0mu^{m'}$}, and a 
\mbox{$\hat D \hskip-1pt\in \hskip-1pt \null^{m'}\mkern-0mu\overline R[x]\mkern0mu^n$}
such that \mbox{$\hat C{\cdot}\hat D = \overline E$} and
\mbox{$\hat D{\cdot}\hat C = \overline{I_{m'}}$}, where the circumflexes  are  reminders 
that the matrices are over~$\overline R[x]$. 
Over the field $\overline R(x)$, we then have  \mbox{$\overline{A}{\cdot}\overline{B} = \overline{E} = 
\hat{C}{\cdot}\hat{D}$},
while \mbox{$\overline{B}{\cdot}\overline{A} = \overline{I_{m}}$} and  
 \mbox{$\hat{D}{\cdot}\hat{C} = \overline{I_{m'}}$}.
Now \mbox{$ \hat{D}{\cdot}(\overline{A} {\cdot} \overline{B}){\cdot}\hat{C}  
= \hat{D}{\cdot}(\hat{C}{\cdot}\hat{D}){\cdot}\hat{C} = \overline{I_{m'}}{\cdot}\overline{I_{m'}}$}
and \mbox{$\overline{B}{\cdot}(\hat{C}{\cdot}\hat{D}){\cdot} \overline{A}  = 
\overline{B}{\cdot}(\overline{A} {\cdot} \overline{B}){\cdot} \overline{A} 
= \overline{I_{m}}{\cdot}\overline{I_{m}}.$} Hence, over the field \mbox{$\overline R(x)$},
the matrices  \mbox{$\hat{D}{\cdot}\overline{A}$}
and \mbox{$\overline{B}{\cdot}\hat{C}$} are mutually inverse, 
so square; thus,  \mbox{$m' = m$}. 
By Lemma~\ref{basics}$(iii)$ above, 
there  exists some invertible \mbox{$U  \in \null^m\mkern-3mu R(x)\mkern0mu^m$} such that
\mbox{$\overline{U} = \hat{D}{\cdot}\overline{A}$} (and \mbox{$\overline{U^{-1}} = \overline{B}{\cdot}\hat{C}$}).  
 Over \mbox{$R(x)$}, set \mbox{$A':= A{\cdot\,}U^{-1}$} and \mbox{$B':= U{\cdot} B$};
then \mbox{$A'{\cdot}B' = E$}  and   \mbox{$B'{\cdot}A' = I_m$}. 
Over~\mbox{$\overline R(x)$},  we then have 
\mbox{$\overline{A'} 
= \overline{A}{\cdot}\overline{U^{-1}}
= \overline{A}{\cdot} \overline{B}{\cdot}\hat{C} = \hat{C}{\cdot}\hat{D}{\cdot}\hat{C} 
= \hat{C}{\cdot} \overline{I_m} 
= \hat{C}$}  and \mbox{$ \overline{B'} 
=  \overline{U}{\cdot}\overline{B} 
=  \hat{D}{\cdot}\overline{A}{\cdot}\overline{B} 
=  \hat{D}{\cdot}\hat{C}{\cdot}\hat{D}
 = \overline{I_m}{\cdot}\hat{D} = \hat{D}.$} 
Thus,  \mbox{$\overline{A'}$}, \mbox{$\overline{B'}$} are matrices over \mbox{$\overline R[x]$},
which means that  \mbox{$A'$}, \mbox{$B'$}  are matrices over the subring \mbox{$R[x] + \mathfrak{m}{\cdot} R(x)$}
of \mbox{$R(x)$}.  By Lemma~\ref{basics}$(i)$ above,   \mbox{$R(x)= R[x]\oplus R(x)_\circ$} as $R$-modules.  Hence,
 \mbox{$R[x] + \mathfrak{m}{\cdot} R(x) = R[x]\oplus \mathfrak{m}{\cdot}R(x)_\circ$}, and,
 over this subring of $R(x)$,  we  have   \mbox{$A'{\cdot}B' = E$}  and  
 \mbox{$B'{\cdot}A' = I_m$}.  

Set
\mbox{$[\null^m\mkern-3muR[x]\mkern0mu^n{\cdot}A']:= \{[F{\cdot}A'] : F \in \null^m\mkern-3muR[x]\mkern0mu^n\} \subseteq 
\null^m\mkern-3muR[x]\mkern0mu^m$}, where the  polynomial part of \mbox{$F{\cdot}A'$} is taken coordinate-wise. 
We claim that \mbox{$[\null^m\mkern-3muR[x]\mkern0mu^n{\cdot}A'] = \null^m\mkern-3muR[x]\mkern0mu^m$}.
Consider any \mbox{$G \in \null^m\mkern-3muR[x]\mkern0mu^m$};
it suffices to show that \mbox{$G \in  [\null^m\mkern-3muR[x]\mkern0mu^n{\cdot}A']$}.
There exists a monic$\mkern8mu$\mbox{$h \in R[x]$} such that \mbox{$h{\cdot}B' \in \null^m\mkern-3muR[x]\mkern0mu^n$}.
Now   \mbox{$[\null^m\mkern-3muR[x]\mkern0mu^n{\cdot}A'] \ni [(G{\cdot}(h{\cdot}B')){\cdot}A'] 
= [G{\cdot} h{\cdot}I_m ] = G{\cdot}h$}; 
hence, \mbox{$[\null^m\mkern-3muR[x]\mkern0mu^n{\cdot}A'] \supseteq \null^m\mkern-3muR[x]\mkern0mu^m{\cdot}h$}.
Here, 
 \mbox{$\null^m\mkern-3muR[x]\mkern0mu^m/(\null^m\mkern-3muR[x]\mkern0mu^m{\cdot}h)$}  is  finitely 
$R$-generated, as it is a finitely generated  \mbox{$(R[x]/(R[x]{\cdot}h))$}-mod\-ule  and $h$ is monic.
Hence, \mbox{$\null^m\mkern-3muR[x]\mkern0mu^m/[\null^m\mkern-3muR[x]\mkern0mu^n{\cdot}A']$} too is  
finitely $R$-generated.   By Lemma~\ref{basics}$(iv)$ above, to prove 
  \mbox{$[\null^m\mkern-3muR[x]\mkern0mu^n{\cdot}A'] = \null^m\mkern-3muR[x]\mkern0mu^m$},  it  suffices to show  
 \mbox{$G \in \mathfrak{m}{\cdot}\null^m\mkern-3muR[x]\mkern0mu^m  {+} [\null^m\mkern-3muR[x]\mkern0mu^n{\cdot}A']$}.
In \mbox{$\overline R[x]$},

\vspace{-5mm}

\begin{align*}
\overline G = \overline{G{\cdot}I_m} = \overline{G{\cdot}B'{\cdot}A'} = 
\overline{G{\cdot}B'}{\cdot}\overline {A'} = \overline{[G{\cdot}B']}{\cdot}\overline {A'}
= \overline{[G{\cdot}B'] {\cdot} A'} = \overline{[[G{\cdot}B'] {\cdot} A']}.
\end{align*}

 \vspace{-1mm}

\noindent Hence, each coordinate  of  \mbox{$G{-}[[G{\cdot}B']{\cdot}A']$} lies in
\mbox{$R[x] \cap \mathfrak{m}{\cdot}R(x) = \mathfrak{m}{\cdot}R[x]$}.  This shows \linebreak that
\mbox{$G = (G{-}[[G{\cdot}B']{\cdot}A']){+}[[G{\cdot}B']{\cdot}A'] \in 
\mathfrak{m}{\cdot}\null^m\mkern-3muR[x]\mkern0mu^m {+} [\null^m\mkern-3muR[x]\mkern0mu^n{\cdot}A']$},
as desired; thus, we have \mbox{$[\null^m\mkern-3muR[x]\mkern0mu^n{\cdot}A'] =
 \null^m\mkern-3muR[x]\mkern0mu^m$}.  The  point of this paragraph is that 
 we may now fix an \mbox{$F'  \in \null^m\mkern-3muR[x]\mkern0mu^n$} such that  
  \mbox{$[F'{\cdot}A'] = I_m$}.  By symmetry, 
we may  fix a \mbox{$G'  \in \null^n\mkern-3muR[x]\mkern0mu^m$} such that
  \mbox{$[B'{\cdot}G'] = I_m$}. 

Notice that each coordinate 
  of \mbox{$F'{\cdot}A' {-} I_m$}
 lies in  \mbox{$(R[x]\,{\oplus}\,\mathfrak{m}{\cdot}R(x)_\circ) \cap R(x)_\circ = \mathfrak{m}{\cdot} R(x)_\circ$};
thus,  \mbox{$\operatorname{Det}(F'{\cdot}A') \in 1{+} \mathfrak{m}{\cdot} R(x)_\circ$}.
Now \mbox{$ 1{+} \mathfrak{m}{\cdot} R(x)_\circ =
 \{\frac{f}{g} \in R(x) : f, g\,\,{\rm monic}, \overline f = \overline g\}$},   which
is a multiplicative group.  Over \mbox{$R[x] \oplus \mathfrak{m}{\cdot}R(x)_\circ$}, 
  \mbox{$F'{\cdot}A'$} is therefore invertible, and 
we may form  \mbox{$A'':= A'{\cdot}\,(F'{\cdot}A')^{-1}$} and 
\mbox{$B'':=  (F'{\cdot}A'){\cdot}B'$}; then   \mbox{$A''{\cdot}B''   =   E$} and \mbox{$B''{\cdot}A'' = I_m$}.
As before, we may   fix a
\mbox{$G'' \in \null^n\mkern-3muR[x]\mkern0mu^m$} such that
  \mbox{$[B''{\cdot}G''] = I_m$}.
Then
\mbox{$B'' =   F'{\cdot}A' {\cdot}B'  = F'{\cdot}E\in \null^m\mkern-3muR[x]\mkern0mu^n$}
and   \mbox{$A'' = A''{\cdot}I_m =
 A''{\cdot}[B''{\cdot}G''] =   A'' {\cdot} B''{\cdot} G'' = 
E {\cdot} G'' \in \null^n\mkern-3muR[x]\mkern0mu^m$}, as desired.  
\end{proof}


By Quillen's extension theorem
\cite[Theorem\,1]{Quillen1976}, \cite[Theorem]{Roitman1979}, \linebreak if $R$ is a commutative ring 
and $P$ is a finitely presented \mbox{$R[x]$}-module such that,
for each \mbox{$\mathfrak{m} \in \operatorname{Max}(R)$}, 
 \mbox{$R_{\mathfrak{m}}[x]{\otimes}_{R[x]}P$} is \mbox{$R_{\mathfrak{m}}[x]$}-free  (or, more generally,
\mbox{$R_{\mathfrak{m}}[x]$}-isomorphic to
  \mbox{$R_{\mathfrak{m}}[x]{\otimes\hspace{-1pt}}_{R_{\mathfrak{m}}}\hspace{-2pt} P_{\mathfrak{m}}$}
for some \mbox{$R_{\mathfrak{m}}$}-module \mbox{$P_{\mathfrak{m}}$}),
then $P$ is \mbox{$R[x]$}-isomorphic to \mbox{$R[x]{\otimes_R} P\mkern1mu'$} for some \mbox{$R$}-module $P\mkern1mu'$. 
Here, all that we shall  prove is the domain   case of the result obtained by combining Quillen's extension theorem
 with Theorem~\ref{th:Horr} above. 
 
\begin{theorem}\label{th:q}  If $R$ is a domain   
and  $P$ is a finitely
generated, projective $R[x]$-module such  that
\mbox{$R_{\mathfrak{m}}(x){\otimes}_{R[x]}P\,$} is  \mbox{$R_{\mathfrak{m}}(x)$}-free
  for \vspace{-2.4mm} each \mbox{$\mathfrak{m} \in \operatorname{Max}(R)$}, then
\mbox{$P \underset{R[x]}{\simeq} R[x]\otimes_R\mkern-3muP\mkern1mu'$}  for some finitely  generated, 
projective $R$-module $P\mkern1mu'$.
\end{theorem}

\begin{proof}[Proof  $($Quillen, Vaserstein, Roitman$\mkern3mu)$] 
Set  \mbox{$R[x,y]:= R[x][y] = R[y][x]$}. 
 For   each   \mbox{$f  \in R[x,y]$}, we write  
\mbox{$(y \mapsto f){:\,}R[x,y] \to R[x,y]$}, \mbox{$g \mapsto g^{(y\mapsto f)}$}, to
denote the unique ring endomorphism of $R[x,y]$ which sends $y$ to $f$ and fixes 
each element of $R[x]$.  We define   \mbox{$(x \mapsto f)$} analogously. 
We extend these maps to matrices coordinate-wise.
 
 By 
Review~\ref{rev}$(iii)$ above,
there exist an \mbox{$n \in \NN$} 
and an idempotent  
\mbox{$E \in \null^n\mkern-3muR[x]\mkern0mu^n$} such that  \mbox{$P \underset{R[x]}{\simeq} R[x]\mkern2mu^n{\cdot}E$}.
Set \mbox{$J:= \{j \in R \mid E^{(x\mapsto x+j{\cdot}y)}  \underset{R[x,y]}{\sim}\\ E\}$}. \vspace{.5mm}
 
We  now show that $J$ is an ideal of $R$.  Clearly \mbox{$0 \in J$}.  
Consider any $j, j' \in J$. Here,   \mbox{$E^{(x\mapsto x+j{\cdot}y)} \underset{R[x,y]}{\sim} E$},
and    applying   \mbox{$(x \mapsto x+j'{\cdot}y)$} yields
 \mbox{$E^{(x \mapsto x +j'{\cdot}y+j{\cdot}y)} \underset{R[x,y]}{\sim}
E^{(x \mapsto x + j'{\cdot}y)}$}.  Since    \mbox{$E^{(x\mapsto x+j'{\cdot}y)} \underset{R[x,y]}{\sim} E$},
we see   that  \mbox{$E^{(x \mapsto x +j'{\cdot}y+j{\cdot}y)} \underset{R[x,y]}{\sim}
E$}, whence \mbox{$j+j' \in J$}.
Now consider any \mbox{$j \in J$} and \mbox{$r \in R$}.  Since $E$ is a matrix over $R[x]$,
 applying \mbox{$(y \mapsto r{\cdot}\,y)$} to \mbox{$E^{(x\mapsto x+j{\cdot}y)}\underset{R[x,y]}{\sim} E$} 
   yields 
 \mbox{$E^{(x \mapsto x + j{\cdot}r{\cdot}y)}\underset{R[x,y]}{\sim} E$}, whence \mbox{$j{\cdot}r \in J$}.  
Hence,  $J$ is an ideal of~$R$.

Consider an arbitrary   \mbox{$\mathfrak{m} \in \operatorname{Max}(R)$}. 
By hypothesis, 
\mbox{$R_{\mathfrak{m}}(x)  \otimes_{R_\mathfrak{m}[x]} ( R_\mathfrak{m}[x] {\otimes}_{R[x]} P)$} 
is  \linebreak\mbox{$R_{\mathfrak{m}}(x)$-free}.  Then,  by Theorem~\ref{th:Horr} above,  
\mbox{$R_\mathfrak{m}[x]{\otimes}_{R[x]}P$} is \mbox{$R_{\mathfrak{m}}[x]$-free}.  
We may view $R$ as a subring of~\mbox{$R_{\mathfrak{m}}$}, since $R$~is a domain.   
By Review~\ref{rev}$(iv)$ above, 
 there exist matrices \linebreak $I$, $A$, and $B$ over \mbox{$R_{\mathfrak{m}}[x]$}  such that $I$ is an  identity matrix, 
\mbox{$A{\cdot}B = E$}, and \mbox{$B{\cdot}A = I$}.  
We  extend our ring-endo\-mor\-phism notation from \mbox{$R[x,y]$} to \mbox{$R_{\mathfrak{m}}[x,y]$},
and set \mbox{$C:= A{\cdot}B^{(x \mapsto x{+}y)}$} and
\mbox{$D:= A^{(x \mapsto x{+}y)}{\cdot}B$}.  
There exists some \mbox{$r \in R{-}\mathfrak{m}$} such that \mbox{$r{\cdot}C$} and \mbox{$r{\cdot}D$} 
are matrices over \mbox{$R[x,y]$}.
On viewing $C$ as a polynomial in $y$ with coefficients that are matrices over \mbox{$R_{\mathfrak{m}}[x]$},
we see that the constant term is \mbox{$C^{(y \mapsto 0)} = A{\cdot}B = E$}, which is a matrix over  \mbox{$R[x]$}.  
Hence,  \mbox{$C^{(y \mapsto r{\cdot}y)}$} is a matrix  over \mbox{$R[x,y]$}.  Similarly,
  \mbox{$D^{(y \mapsto r{\cdot}y)}$} is a matrix  over \mbox{$R[x,y]$}.  Over \mbox{$R_{\mathfrak{m}}[x,y]$}, 
\begin{align*}
C^{(y \mapsto r{\cdot}y)}{\cdot}D^{(y \mapsto r{\cdot}y)} &= 
A{\cdot}B^{(x \mapsto x{+}r{\cdot}y)}{\cdot} A^{(x \mapsto x{+}r{\cdot}y)}{\cdot}B = 
A{\cdot}(B{\cdot}A)^{(x \mapsto x+r{\cdot}y)} {\cdot}B = A{\cdot} I {\cdot}B = E,
\\D^{(y \mapsto r{\cdot}y)}{\cdot}C^{(y \mapsto r{\cdot}y)} 
&= A^{(x \mapsto x{+}r{\cdot}y)}{\cdot}B{\cdot}A{\cdot} B^{(x \mapsto x{+}r{\cdot}y)} =
A^{(x \mapsto x{+}r{\cdot}y)}{\cdot}I{\cdot} B^{(x \mapsto x{+}r{\cdot}y)} = E^{(x \mapsto x + r{\cdot}y)}.
\end{align*}

\vspace{-.2cm}

\noindent  Hence,
  \mbox{$E \underset{R[x,y]}{\sim} E^{(x \mapsto x + r{\cdot}y)}$}; that is, \vspace{-1mm} 
\mbox{$r \in J$}.  Since   \mbox{$r \not\in \mathfrak{m}$}, we have \mbox{$J \nsubseteq \mathfrak{m}$}. 
Since \mbox{$\mathfrak{m} \in \operatorname{Max}(R)$} is arbitrary,  \mbox{$1 \in J$}; that is,
\mbox{$E^{(x \mapsto x + y)} \underset{R[x,y]}{\sim} E$}.
Thus, \vspace{-1mm}  there exist matrices \mbox{$A_0$}, \mbox{$B_0$}  over \mbox{$R[x,y]$}
such that \mbox{$A_0{\cdot}B_0 = E^{(x \mapsto x+y)}$} and
\mbox{$B_0{\cdot}A_0 = E$}.   
Then \mbox{$A_0^{(y \mapsto -x)}$} and {$B_0^{(y \mapsto -x)}$} are matrices over~\mbox{$R[x]$},
\mbox{$A_0^{(y \mapsto -x)}{\cdot}B_0^{(y \mapsto -x)} = 
(E^{(x \mapsto x+y)})^{(y \mapsto -x)} = E^{(x \mapsto 0)}$}, and 
\mbox{$B_0^{(y \mapsto -x)}{\cdot}A_0^{(y \mapsto -x)} = E^{(y \mapsto -x)} = E$}.  
Thus, \mbox{$E^{(x \mapsto 0)} \underset{R[x]}{\sim} E$},  
which   is the desired conclusion, by Review~\ref{rev}$(iv)$ above.  
\end{proof}

We now prove the result of \cite[Theorem $4'$]{Quillen1976} and  \cite[Theorem 3]{Suslin1976}.
 
\newpage

\begin{qs}\label{th:qs} Let \mbox{$n \in \NN$}, let $K$ be a Dedekind \mbox{domain},
let  $R$ denote \mbox{$K[x_1,\ldots,x_n]$}, and let $P$~be a finitely generated,
projective  \mbox{$R$}-mod\-ule.  Then \mbox{$P \underset{R}{\simeq} R \otimes_{K}\mkern-3muP_0$}
for some finitely generated, projective $K$-module $P_0$.  If $K$ is a PID, then
$P_0$ is $K$-free and $P$ is  \mbox{$R$}-free.
\end{qs}

\noindent \textit{Proof $($Quillen$)$}.   
We argue by induction on~$n$.  As the result is clear for \mbox{$n=0$},   
we may assume that \mbox{$n \ge 1$} and that the result holds  with \mbox{$n{-}1$} in place of~$n$.

Set \mbox{$X:= \{x_2,\ldots, x_n\}$}. 
Consider an arbitrary \mbox{$\mathfrak{m} \in \operatorname{Max}(K[X])$}.  There then exists some   
 \mbox{$\mathfrak{m}' \in \operatorname{Max}(K)$} such that $\mathfrak{m} \cap K \subseteq \mathfrak{m}'$.
By hypothesis,  $P$ is a finitely generated, projective $K[x_1][X]$-mod\-ule. 
We view  \mbox{$K[x_1][X] \subseteq K_{\mathfrak{m}'}(x_1)[X]$}.
 Then
 \mbox{$K_{\mathfrak{m}'}(x_1)[X]\otimes_{K[x_1][X]}P$}  is a finitely generated, projective
\mbox{$K_{\mathfrak{m}'}(x_1)[X]$}-mod\-ule; 
it is \mbox{$K_{\mathfrak{m}'}(x_1)[X]$}-free by
 the induction hypothesis, since \mbox{$K_{\mathfrak{m}'}(x_1)$} is a PID by Lemma~\ref{basics}$(v)$ above,
since $K_{\mathfrak{m}'}$ is a local PID.
 We now view \mbox{$K_{\mathfrak{m}'}(x_1)[X] \subseteq K[X]_{\mathfrak{m}}(x_1)$}. 
Then  \mbox{$K[X]_{\mathfrak{m}}(x_1)\otimes_{K_{\mathfrak{m}'}(x_1)[X]}\hskip-1.8pt(K_{\mathfrak{m}'}
(x_1)[X]\otimes_{K[x_1][X]}P)$} is \mbox{$K[X]_{\mathfrak{m}}(x_1)$}-free;
that is,  \mbox{$K[X]_{\mathfrak{m}}(x_1)\otimes_{K[X][x_1]}\hskip-1.8ptP$} is
\mbox{$K[X]_{\mathfrak{m}}(x_1)$}-free.  Since \mbox{$\mathfrak{m} \in \operatorname{Max}(K[X])$} is arbitrary,
Theorem~\ref{th:q} above now shows that   
 \mbox{$P \mkern-5mu \underset{K[X][x_1]}{\simeq} K[X][x_1] \otimes_{K[X]} \mkern-3mu P\mkern1mu'$}
for some finitely generated, projective \mbox{$K[X]$}-module $P\mkern1mu'$.  
By the induction hypothesis   again,  
 \mbox{$P\mkern1mu'\mkern-5mu\underset{K[X]}{\simeq} K[X] \otimes_K P_0$} for some finitely generated, projective \mbox{$K$}-module 
$P_0$, and then
  \mbox{$P\mkern-5mu\underset{K[X][x_1]}{\simeq} K[X][x_1]\otimes_K\mkern-3muP_0$}. \qed


\bibliographystyle{plain}

\end{document}